\documentclass[12pt,norsk]{amsart}
\usepackage[T1]{fontenc}
\usepackage[utf8]{inputenc}
\usepackage{amsthm,amsfonts,amsmath,mathrsfs,amssymb}
\usepackage{dsfont}
\usepackage{fourier}
\usepackage{times,a4wide}

\newtheorem*{theorem*}{Hal\'{a}sz's theorem}
\newtheorem{theorem}{Theorem}

\newtheorem{lemma}{Lemma}

\newcommand{\Real}{\operatorname{Re}}
\newcommand{\Imag}{\operatorname{Im}}

\newenvironment{proof*}{\textbf{\emph{Proof of }}}{\hspace*{\fill} $\square$}
\author[\'{E}ric Sa{\"i}as]{\'{E}ric Sa{\"i}as}
\address{Sorbonne Universit\'{e}, LPSM, 4 Place Jussieu
F-75005 Paris, France}
\email{eric.saias@upmc.fr}

\author[Kristian Seip]{Kristian Seip}
\address{Department of Mathematical Sciences \\ Norwegian University of Science and Technology \\ NO-7491 Trondheim \\ Norway}
\email{kristian.seip@ntnu.no}

\thanks{The research of the second author was supported in part by Grant 275113 of the Research Council of Norway.}

\begin{document}


\title[A footnote to a theorem of Hal\'{a}sz]{A footnote to a theorem of Hal\'{a}sz}
 

\maketitle
 
 \begin{abstract}
We study multiplicative functions $f$ satisfying $|f(n)|\le 1$ for all $n$, the associated Dirichlet series $F(s):=\sum_{n=1}^{\infty} f(n) n^{-s}$, and the summatory function $S_f(x):=\sum_{n\le x} f(n)$. Up to a possible trivial contribution from the numbers $f(2^k)$, $F(s)$ may have at most one zero or one pole on the one-line, in a sense made precise by Hal\'{a}sz. We estimate $\log F(s)$ away from any such point and show that if $F(s)$ has a zero on the one-line in the sense of Hal\'{a}sz, then
$|S_f(x)|\le (x/\log x) \exp\big(c\sqrt{\log \log x}\big)$ for all $c>0$ when $x$ is large enough. This bound is best possible.

 \end{abstract}

Hal\'{a}sz  obtained in \cite{Ha1, Ha} some fundamental results on the mean values of multiplicative functions $f$ subject to the restriction $|f(n)|\le 1$ for all nonnegative integers $n$. We denote this class of functions by $\mathcal{M}$ and set 
\[ S_f(x):=\sum_{n\le x} f(n) \quad \text{and} \quad F(s):=\sum_{n=1}^{\infty}\frac{f(n)}{n^s}, \] 
where the latter series converges absolutely for $\sigma:=\Real s >1$. Following Montgomery \cite{M1}, we have the following.
\begin{theorem*}
Suppose that $f$ belongs to $\mathcal{M}$. Then for every real $t$ with at most one exception, we have
\begin{equation} \label{eq:except} F(\sigma+it)=o\left(\frac{1}{\sigma-1} \right), \quad \sigma \searrow 1. \end{equation}
If there exists an exceptional $t=t_0$ for which \eqref{eq:except} does not hold, then
\begin{equation} \label{eq:exceptional} F(\sigma+it_0) \asymp \frac{1}{\sigma-1}, \quad 1<\sigma \le 2. \end{equation}
Moreover, the following three assertions are equivalent:
\begin{itemize}
\item[(i)] $S_f(x)=o(x), \quad x\to \infty$; 
\item[(ii)] For every real $t$, $F(\sigma+it) = o\big(1/(\sigma-1)\big)$ when $\sigma\searrow 1$;
\item[(iii)] For every real $t$, we have 
\[ \sum_p \frac{1-\Real\big( f(p) p^{-it} \big)}{p}=+ \infty \quad
\text{or} \quad  f(2^k)=-2^{ikt} \quad \text{for all $k\ge 1$}. \]
\end{itemize}
\end{theorem*} 
The three equivalent assertions (i), (ii), (iii) give a more precise statement about the case $S_f(x)=o(x)$ than what is found in the usual ``textbook version'' of Hal\'{a}sz's theorem; see for example \cite[Sect. 4.3]{Te}. All the statements above can still be extracted from Satz 1' of \cite{Ha1}. The second alternative in item (iii) accounts for a trivial reason for having $F(\sigma+it)=o(1/(\sigma-1))$ when $\sigma\searrow 1$, namely the existence of $t$ such that
\[ \sum_{k\ge 0} \frac{f(2^k)}{2^{k(\sigma+it)}}=\frac{2^{\sigma}-2}{2^{\sigma}-1}. \]
In our first theorem, we exclude this possibility by considering the subclass $\mathcal{M}_2$ of $\mathcal{M}$ consisting of $f$ for which $f(2^k)=0$ for every $k\ge 1$. 

We may think of the exceptional case $t=t_0$ in Hal\'{a}sz's theorem as the assertion that $F(s)$ has a ``simple pole'' at the point $s=1+it_0$. Following \cite[Thm. 2.1]{S}, we find it natural to treat such ``poles'' on equal terms with possible ``zeros'' on the line $\sigma=1$. This allows us to incorporate the following  consequence of the prime number theorem in the first part of the theorem: if there is such a ``zero'' or a ``pole'', there can be no other point of the same kind. This version of Hal\'{a}sz's result also comes with a precise estimate:
\begin{theorem}\label{thm:key}
Suppose that $f$ belongs to $\mathcal{M}_2$. Then for every real $t$ with at most one exception, 
\begin{equation} \label{eq:exc} \lim_{\sigma\searrow 1} \frac{|F(\sigma+it)|^{\varepsilon}}{\sigma-1}=+\infty \end{equation}
for both $\varepsilon=-1$ and $\varepsilon=1$. In fact, if there exists a pair $(\varepsilon,t)=(\varepsilon_0,t_0)$ in
$\{-1,1\}\times \mathbb{R}$ for which \eqref{eq:exc} does not hold, then for $1<\sigma\le 3/2$, 
\[ [F(\sigma+i t_0)]^{\varepsilon_0} \asymp (\sigma-1) \]
and
\begin{equation}\label{eq:bound} \varepsilon_0 \log F(\sigma+it)+ \log \zeta(\sigma+it-it_0) = o\left(\sqrt{\log \frac{1}{\sigma-1}}\right), \end{equation}
uniformly for all real $t$ when $\sigma\searrow 1$.
\end{theorem}
As far as the mean values of $f$ are concerned, the bound in \eqref{eq:bound} is of no interest when $\varepsilon_0=-1$. What matters is then only the behavior of $F(\sigma+it_0)$ when $\sigma\searrow 1$, and we will in particular have that
$|S_f(x)|/x$ tends to a positive limit; see \cite{GS} for precise information about the relation between $F(\sigma+it_0)$ and the mean values $S_f(x)/x$ in the case $\varepsilon_0=-1$. However, when $\varepsilon_0=1$, the estimate in \eqref{eq:bound} yields a sharp improvement of the bound in item (i) of Hal\'{a}sz's theorem.
\begin{theorem}\label{thm:estimate}
Suppose that $f$ belongs to $\mathcal{M}$. If there exists a real $t_0$ such that 
\begin{equation} \label{eq:zero} \sum_p \frac{1+\Real \big(f(p) p^{-it_0}\big)}{p} < \infty, \end{equation}
then
\begin{equation} \label{eq:upper} \limsup_{x\to \infty} \frac{|S_f(x)| \log x}{x  \exp\left(c\sqrt{\log\log x}\right)}=0
\end{equation}
for every constant $c>0$. Conversely, if $\kappa:[3,\infty)\to \mathbb{R}^+$ satisfies $\kappa(x)= o\big(\sqrt{\log \log x}\big)$ when $x\to \infty$, then there exists an $f$ in $\mathcal{M}$ such that \eqref{eq:zero} holds for $t_0=0$ and
\begin{equation} \label{eq:sharp} \limsup_{x\to \infty} \frac{|S_f(x)| \log x}{x  \exp\left(\kappa(x)\right)}=\infty. \end{equation}
\end{theorem}
We obtain \eqref{eq:upper} as an immediate consequence of \eqref{eq:bound} and a celebrated elucidation of item (i) of Hal\'{a}sz's theorem, expressed in terms of the size of $|F(s)|$ close to the $1$-line.\footnote{To this end, we use the classical fact that
$1/\zeta(\sigma+it) \ll \log (|t|+2)$ 
holds uniformly for $\sigma\ge 1$ and real $t$.} This result also stems from work of Hal\'{a}sz \cite{Ha1, Ha}; see Montgomery's paper \cite{M}, Tenenbaum's book \cite[Sec. III.4.3]{Te}, or the recent paper \cite{GHS}. We will therefore give below only the proof of the second part of Theorem~\ref{thm:estimate}.


Before proving our two theorems, we establish the following lemma.
\begin{lemma}\label{lem:basic}
Let $f(p)$ be a sequence of numbers satisfying $|f(p)|\le 1$. Suppose that there exist $\varepsilon_0$ in $\{-1,1\}$ and a real number $t_0$ such that 
\begin{equation} \label{eq:either} \sum_{p} \frac{1+\varepsilon_0 \Real \left(f(p) p^{-it_0}\right)}{p} < \infty . \end{equation} 
Then
\[  \varepsilon_0 \sum_{p} \frac{f(p)}{p^s} + \log \zeta(s-it_0)  = o\left(\sqrt{\log \frac{1}{\sigma-1} }\right) \]
uniformly for $s=\sigma+it$, $\sigma\searrow 1$, and real $t$.
\end{lemma}

\begin{proof}[Proof of Lemma~\ref{lem:basic}]
Our initial assumption is that \eqref{eq:either} holds for either $\varepsilon_0=-1$ or $\varepsilon_0=1$.
Writing $\varepsilon_0 f(p)p^{-it_0}=:-|f(p)| e^{i\theta_p}$ with $-\pi < \theta_p\le \pi$, we see that
\[ 1+\varepsilon_0 \Real \big(f(p) p^{-it_0} \big)=  1-|f(p)| +  |f(p)|(1-\cos \theta_p) \ge |f(p)|(1-\cos \theta_p)\ge \frac{|f(p)|}{2\pi} \theta_p^2,  \]
so that \eqref{eq:either} implies that
\begin{equation} \label{eq:psum}  \sum_p \frac{|f(p)| \theta_p^2}{ p}<\infty. \end{equation}

We may now write
\begin{align} \varepsilon_0 \sum_p \frac{f(p)}{p^s} & = \sum_p \frac{\Real \left(\varepsilon_0 f(p) p^{-it_0}\right)}{p^{s-it_0}} + i \sum_p 
\frac{\Imag \left(\varepsilon_0 f(p) p^{-it_0}\right)}{p^{s-it_0}} \nonumber \\
& = - \sum_p \frac{1}{p^{s-it_0}}-i  \sum_{p} \frac{|f(p)| \sin \theta_p }{p^{s-it_0}} +O(1)\nonumber  \\
& = - \log \zeta (s-it_0) - i  \sum_{p} \frac{|f(p)| \sin \theta_p }{p^{s-it_0}} +O(1), \label{eq:im} \end{align}
which holds uniformly for $\sigma>1$.  By Mertens's theorem for the sum $\sum_{p\le x} 1/ p$, the Cauchy--Schwarz inequality, and \eqref{eq:psum}, 
\begin{align*}  \sum_{p} \frac{|f(p) \sin \theta_p |}{p^{\sigma}}
& \le \log\log \frac{1}{\sigma-1} +O(1)+\sum_{p\ge 1/(\sigma-1)} \frac{|f(p) \sin \theta_p |}{p^{\sigma}} \\ 
& \le   \log\log \frac{1}{\sigma-1} +O(1) + \left(\sum_{p} p^{1-2\sigma}\right)^{1/2} \left(\sum_{p\ge 1/(\sigma-1)} \frac{|f(p)| \theta_p^2}{p}\right)^{1/2} \\
& = 
o\left( \sqrt{\log \frac{1}{\sigma-1} } \right)\end{align*} 
when $\sigma\searrow 1$. Plugging this estimate into \eqref{eq:im}, we obtain the desired bound.
\end{proof}

\begin{proof}[Proof of Theorem~\ref{thm:key}]
Since $\zeta(s)$ has a simple pole at $s=1$, is otherwise analytic, and has no zero on $\sigma=1$, the first part of Theorem~\ref{thm:key} is an immediate consequence of the second part. To prove the latter assertion, we assume that 
\[ \frac{|F(\sigma+it_0)|^{\varepsilon_0}}{\sigma-1} \]
does not tend to $+\infty$ when $\sigma\searrow 1$
for some pair $(\varepsilon_0,t_0)$ in $\{-1,1\} \times \mathbb{R}$. By assumption, $f$ is in $\mathcal{M}_2$, whence 
 \[ F(s):=\prod_{p>2} \sum_{k=0}^\infty  \frac{f(p^k)}{p^{ks}}. \]
For $p\ge 3$, we have $|\sum_{k=1}^{\infty} f(p^k) p^{-ks} |\le 1/2$. We may therefore infer that
\[ \log F(s) =\sum_p f(p) p^{-s} + O(1) \]
and hence 
\[ \log |F(s)|  =\sum_p \Real\left(f(p) p^{-s}\right) + O(1) \]
when $\sigma>1$. It follows from this and the fact that $\zeta(s)$ has a simple pole at $s=1$ that
\[  \frac{|F(s)|^{\varepsilon_0}}{\sigma-1} \asymp \zeta(\sigma) |F(s)|^{\varepsilon_0}  
\asymp \exp\left\{  \sum_{p} \frac{1+\varepsilon_0 \Real \left(f(p) p^{-it_0}\right)}{p^{\sigma}} \right\} \]
when $1<\sigma \le 3/2$. By monotone convergence,
\[  \lim_{\sigma\searrow 1} \sum_{p} \frac{1+\varepsilon_0 \Real \left(f(p) p^{-it_0}\right)}{p^{\sigma}}= \sum_{p} \frac{1+\varepsilon_0 \Real \left(f(p) p^{-it_0}\right)}{p}. \]
By assumption, this limit is not $+\infty$, and hence
we may apply Lemma~\ref{lem:basic} to conclude.
\end{proof}

\begin{proof}[Proof of the second part of Theorem~\ref{thm:estimate}] We will assume that every $f$ in $\mathcal{M}$ for which \eqref{eq:zero} holds with $t_0=0$, satisfies
\begin{equation} \label{eq:assume} |S_f(x)|\ll \frac{x}{\log x} \exp\big(\kappa(x) \big), \end{equation}
and show that this leads to a contradiction.

We may clearly assume that $\kappa(x)$ is a continuous function. It is also plain that $\kappa(x)$ may be assumed to be nondecreasing and that $\kappa(x)/\sqrt{\log \log x}$ may be taken to be a nonincreasing function. Indeed, if $\kappa(x)$ failed to be nondecreasing, then we could use instead 
$\kappa_0(x):=\max_{3\le y \le x} \kappa(x)$; should moreover $\kappa_0(x)/\sqrt{\log\log x}$ fail to be nonincreasing, then we could replace it by
\[ \kappa_1(x):=\sqrt{\log\log x} \max_{y\ge x} \frac{\kappa_0(y)}{\sqrt{\log\log y}}, \]
which would still be a nondecreasing function being $o(\sqrt{\log\log x})$ when $x\to \infty$.

By partial summation, we have for every $1<\sigma\le 3/2$ and say $|t|\le 1$,
\begin{align*} |F(\sigma+it)| & \le 1+2 \int_{3}^{\infty} \big|S_f(y)\big| y^{-\sigma-1} dy \ll \int_{3}^{\infty}  \frac{e^{\kappa(y)}}{y^{\sigma} \log y}  dy \\
&  \le \exp\left(\kappa\big(e^{\frac{1}{\sigma-1}}\big)\right) \int_{3}^{e^{1/(\sigma-1)}} \frac{dy}{y \log y} + \int_{e^{1/(\sigma-1)}}^{\infty}  
\frac{e^{-\log\log y+\kappa(y)}}{y^{\sigma} \log y} dy. \end{align*}
Since $\kappa(y)/\sqrt{\log\log y} $ is a nonincreasing function, the function $\log\log y-\kappa(y)$ is eventually increasing, whence the above computation leads to the bound
\[ |F(\sigma+it)| \ll \exp\left(\kappa\big(e^{\frac{1}{\sigma-1}}\big)\right)\Big( \log\frac{1}{\sigma-1} +1/e \Big). \]
We may write this more succinctly as 
\begin{equation}\label{eq:below}  |F(\sigma+it)|\le \exp\left(\alpha\Big(e^{\frac{1}{\sigma-1}}\Big) \sqrt{\log \frac{1}{\sigma-1}}\right), \end{equation}
where $\alpha:[3,\infty)\to (0,\infty)$ is a nonincreasing function satisfying $\alpha(x)\to 0$ when $x\to \infty$.

We now choose a sequence of positive numbers $x_j$, growing so rapidly that $x_j^{\log x_j}<x_{j+1}$ for every $j\ge 1$ and the sequence \[ a_j:=\sqrt{\alpha\big(x_j^{\log x_j}\big)}\] is in $\ell^2$. We then set
\[ \theta_p:=\begin{cases} \frac{a_j}{\sqrt{\log\log p}},  & x_j\le p < x_j^{\log x_j}, \ \ -\Real (i p^{-i}) \ge 1/2 \\
                 0, & \text{otherwise}. \end{cases} \]
 We find that 
  \[ \sum_{p}  \frac{|\theta_p|^2}{p} \le \sum_{j=1}^{\infty} \frac{a_j^2}{\log\log x_j} \sum_{p \le  x_j^{\log x_j}} \frac{1}{p}\ll  \sum_{j=1}^{\infty} a_j^2,  \]
 where we in the last step used Mertens's theorem for the sum $\sum_{p\le x} 1/p$.  Hence $\sum_{p}  |\theta_p|^2/p<\infty$ by our choice of the sequence $a_j$. 
  Setting $f(p):=-e^{i\theta_p}$ and using Taylor's theorem to write
  \[ f(p)=-1-i\theta_p+O(\theta_p^2), \]
  we infer from this that
 \[ \Real \sum_{p} f(p) p^{-s} = - \Real \log \zeta(s)-\Real i\sum_p \theta_p p^{-s} +O(1).\] 
 It does not matter how we define $f$ for higher prime powers, but for definiteness, let us require that $f$ be completely multiplicative. Setting $\sigma=1+1/(\log x_j)^2 $ and $t=1$, we then get
 \begin{align*} \Real \log F(1+1/(\log x_j)^2+i) & = -\Real i \sum_p \theta_p p^{-1-1/(\log x_j)^2-i} +O(1) \\
 & \gg \frac{a_j}{\sqrt{\log\log x_j}}  \sum_{x_j\le p \le x_j^{\log x_j}} \frac{1}{p} +O(1) \gg a_j\sqrt{\log\log x_j}. \end{align*}
But choosing the same $\sigma=1+1/(\log x_j)^2$ and $t=1$ in \eqref{eq:below}, we reach the bound
\[ \sqrt{\alpha\big(x_j^{\log x_j}\big)}\gg 1, \]
contradicting that $\alpha(x) \searrow 0$ when $x\to \infty$, which, as observed above, is a consequence of our assumption that \eqref{eq:assume} holds for all $f$ in $\mathcal{M}$ for which \eqref{eq:zero} is true. 
\end{proof}

\section*{Acknowledgement}

We would like to thank Michel Balazard for some pertinent remarks on the subject matter of this note. We are also grateful to the anonymous referee for a careful review that helped us clarify some essential points.

\end{document}